\documentclass[11pt]{article}
\usepackage{amsmath}
\usepackage{amssymb}
\usepackage{latexsym}
\usepackage{theorem}
\usepackage{mathrsfs}
\usepackage[all]{xy}
\usepackage{dave2}
\usepackage{stmaryrd}
\usepackage{hyperref}

\DeclareMathOperator{\tails}{tails}
\DeclareMathOperator{\seq}{seq}
\DeclareMathOperator{\map}{map}

\title{Rational $\zz_p$-Equivariant Spectra}
\author{David Barnes}
\date{November 26, 2010}

\begin{document}
\maketitle

\begin{abstract}
\noindent
We find a simple algebraic model for rational $\zz_p$-equivariant spectra, via
a series of Quillen equivalences. This model, along with an Adams short exact sequence,
will allow us to easily perform constructions and calculations.
\end{abstract}

\pdfbookmark[1]{Contents}{toc}
\tableofcontents
\newpage

\section{Introduction}
Spaces with an action of a topological group $G$ are of great interest to a wide range of mathematicians. A particularly useful set of tools for studying these spaces are $G$-equivariant cohomology theories. To study these cohomology theories, it is helpful to understand the $G$-spectra that represent them. The homotopy theory of spectra is extremely complex, as demonstrated by the stable homotopy groups of spheres. A great deal of this difficulty comes from torsion, hence it is common to work rationally. This corresponds to studying cohomology theories which take values in rational vector spaces.
The study of equivariant spectra up to homotopy is even more demanding, so it is of even greater importance to rationalise in this case. Rationalising preserves much of the interesting behaviour coming from the group, so we are left with a tractable problem - understanding the rational homotopy theory of $G$-spectra - whose solution would be useful in a number of different areas.

A solution to this problem would be twofold, firstly we would want an abelian category $\acal(G)$, called the algebraic model and an equivalence between the homotopy category of $dg \acal(G)$ and 
the homotopy category of rational $G$-spectra. We would also like the algebraic model to capture
homotopical structures like homotopy limits or colimits, so we ask that our solution consists
of a series of Quillen equivalences between the model category of rational $G$-spectra and the model category $dg \acal(G)$. This is analogous to how a derived equivalence of rings gives an equivalence
of their derived categories and preserves more homological structure than just having an equivalence of derived categories. The algebraic model should be explicit and manageable so that constructing objects or maps is straightforward.

Secondly we want to be able to calculate maps in the homotopy category of rational $G$-spectra using the algebraic model. The Quillen equivalences provide us with an Adams spectral sequence relating maps in the homotopy category of $dg \acal(G)$ to maps in the homotopy category of rational spectra. Studying the algebraic model should give us information on this spectral sequence and help with its calculation.

These aims have been completed for finite groups, through the work of Greenlees and May in \cite[Appendix A]{gremay95}, Schwede and Shipley in \cite[Example 5.1.2]{ss03stabmodcat} and these were improved to the level of a monoidal Quillen equivalence in \cite{barnesfinite}. Greenlees and Shipley have been studying the case of a torus: \cite{gre99}, \cite{shi02} and \cite{greshi}. The case of $O(2)$ has also been largely completed in \cite{gre98a} and \cite{barnesdihedral}.

In this paper we fix $G$ to be $\zz_p$, the group of $p$-adic integers. We give a series of Quillen equivalences between rational $\zz_p$-spectra and an algebraic model $\acal(\zz_p)$. Furthermore the Adams spectral sequence in this case takes the form of a short exact sequence, which makes the model particularly suited for calculations.  We also give relate  $\acal(\zz_p)$
to the algebraic model for $\zz_p/p^n$ and $p^n \zz_p$ via algebraic versions of restriction and inflation. This makes it easy to see how these important equivariant functors behave homotopically.

The group $\zz_p$ is a profinite group: a projective limit of finite groups. Such groups occur particularly often in algebraic geometry, algebraic $K$-theory, number theory and chromatic homotopy theory. They are interesting to study, as they have a non-trivial topology, but (as with a finite groups) have all their homotopical information concentrated in degree zero - a key fact for proving the Quillen equivalences that we require. We present in this paper a first step towards classifying rational equivariant cohomology theories for profinite groups by starting with the canonical example - the $p$-adic integers. We can think of this paper as taking the known cases of rational $\zz/p^n$-spectra and taking the inverse limit over all $n$.
With this viewpoint, a particularly exciting prospect is that the methods of this paper could be applied to the task of understanding rational equivariant cohomology theories for general compact Hausdorff groups, all of which occur as inverse limits of compact Lie groups.

Now we introduce the algebraic model for rational $\zz_p$-equivariant spectra, this is
definition \ref{def:algmodel}.
An object $M$ of $\acal(\zz_p)$ consists of a collection of $\qq[\zz_p/p^k]$-modules $M_k$ for $k \geqslant 0$ and a discrete $\zz_p$-module $M_\infty$, with a specified map of $\zz_p$-modules $$\sigma_M \co M_\infty \to  \colim_n \prod_{k \geqslant n} M_k$$
called the \textbf{structure map} of $M$.
A map $f \co M \to N$ in this category is then a collection of $\qq[\zz_p/p^k]$-module maps $f_k \co M_k \to N_k$ and a map of $\zz_p$-modules $f_\infty \co M_\infty \to N_\infty$ such that the square below commutes
$$
\xymatrix{
M_\infty
\ar[r]^{\sigma_M} \ar[d]_{f_\infty} &
\colim_n \prod_{k \geqslant n} M_k
\ar[d]_{[(f_k)]} \\
N_\infty
\ar[r]^{\sigma_M} &
\colim_n \prod_{k \geqslant n} N_k
}
$$

We also obtain an Adams short exact sequence, theorem \ref{thm:adams}.
The inputs to this sequence are the rational homotopy groups of a spectrum $X$,
which give a graded object $\underline{\pi}_*^\acal(\Sigma X)$ of $\acal(\zz_p)$,
via definition \ref{def:homotopyMackey}. Let $[X,Y]_*^{\zz_p}$ be the set of maps in the
homotopy category of rational $\zz_p$-spectra, then there is a short exact sequence
$$
0 \longrightarrow
\ext^\acal_* (\underline{\pi}_*^\acal(\Sigma X), \underline{\pi}_*^\acal(Y))
\longrightarrow
[X,Y]_*^{\zz_p}
\longrightarrow
\hom_\acal(\underline{\pi}_*^\acal(X), \underline{\pi}_*^\acal(Y))
\longrightarrow 0.
$$

The algebraic model for $\zz_p$ is also worth studying in its own right. It is more intricate than those that occur for finite groups, for example, the injective dimension is non-zero.  Furthermore, the algebraic model we obtain in this case is strikingly similar to the algebraic model of \cite{barnesdihedral}, we discuss this in detail in section \ref{sec:relations}.

For technical reasons the Quillen equivalences we have provided are not monoidal functors,
if one were to construct a category of $\zz_p$-spectra in terms of EKMM $S$-modules, it would
be a simple matter to adjust this paper to show that the algebraic model also captures
the monoidal structure of rational $\zz_p$-spectra. With this in hand one can use the algebraic model to study ring spectra and modules over them.

\subsection*{Organisation}

We start by introducing the group $\zz_p$, the model category of $\zz_p$-spaces and the model category of $\zz_p$-spectra in section \ref{sec:ratspectra}. Next we study the rational Burnside ring of $\zz_p$ in section \ref{sec:burnside}. Once this ring is known, we can define the algebraic model $\acal$ and study its properties in section \ref{sec:model}. We show how our algebraic model is equivalent to the category of rational Mackey functors in section \ref{sec:mackey}. We use this to obtain an Adams short exact sequence in section \ref{sec:adams}. The proof that there is a Quillen equivalence between the model category of differential graded objects in $\acal$ and the model category of rational $\zz_p$-spectra is given in section \ref{sec:proof}. We finish by considering the relation of the algebraic model for $\zz_p$ to those for the groups $p^n \zz_p$ and $\zz/p^n$ in section \ref{sec:relations}, where we also compare $\acal$ to the algebraic model of \cite{barnesdihedral}.

\subsection*{Acknowledgements}
Part of this work was completed at the University of Western Ontario (where the author was partially supported by an NSF grant), the remainder was completed at the University of Sheffield (where the author was supported by the EPSRC grant
EP/H026681/1). The author would like to thank the members of the mathematics departments of both universities for their advice, support and assistance.

\section{A model category for rational \texorpdfstring{$\zz_p$}{Z\_p}-spectra}\label{sec:ratspectra}
Let $p$ be a prime and let $\zz/p^n$ denote the integers modulo $p^n$, for $n \geqslant 0$. Let $proj_n \co \zz/p^{n} \to \zz/p^{n-1}$ denote the projection map. The inverse limit of these projection maps is $\zz_p$, the $p$-adic integers, this group comes with maps $\sigma_n \co \zz_p \to \zz/p^n$. The set $\zz_p$ is topologised with the inverse limit topology, so a set $U$ is open in $\zz_p$ if and only if it is a union of sets of the form $\sigma_n^{-1}(U_n)$, where $U_n$ is an open set in $\zz/p^n$. This makes $\zz_p$ into a compact, Hausdorff and totally disconnected space.

The open subgroups of a profinite group are precisely the closed subgroups of finite index. The set of open subgroups of $\zz_p$is the collection of subgroups of form $p^n \zz_p$, for $n \geqslant 0$. The closed subgroups of $\zz_p$ are the open subgroups and the trivial group.

In order to study $\zz_p$-spectra we use \cite{fauskeqpro}, which constructs
a model category of equivariant orthogonal spectra for compact Hausdorff groups. This construction is a generalisation of
\cite{mm02}. We will then localise this model structure to obtain a model category of rational $\zz_p$-spectra. When working with equivariant spectra, we have to choose some collection of subgroups that we are interested in, different choices will change the model category (and the homotopy category) that we obtain. With compact Lie groups, it is common to consider the collection of all closed subgroups, whereas with profinite groups, it is usual to consider the collection of all open subgroups.
the following result comes from \cite[Proposition 2.11]{fauskeqpro}.

\begin{proposition}
There is a cofibrantly generated proper model structure on the category of based $\zz_p$-spaces with weak equivalences those maps $f$ such that $f^H$ is a weak equivalence of spaces, for all open subgroups $H$. Fibrations are those maps $f$ such that $f^H$ is a fibration of spaces for all open subgroups $H$. This model structure is denoted $\zz_p \tscr$.
\end{proposition}
The generating cofibrations $I_{\zz_p \tscr}$ and acyclic cofibrations $J_{\zz_p \tscr}$ are as below.
$$I_{\zz_p \tscr}= \{ (\zz_p/p^n )_+ \smashprod i | i \in I_{\tscr}, n \geqslant 0 \}$$
$$J_{\zz_p \tscr}= \{ (\zz_p/p^n )_+ \smashprod  j | j \in J_{\tscr}, n \geqslant 0 \}$$


\begin{definition}
A \textbf{$\zz_p$-universe} $U$ is a countable infinite direct sum $U = \oplus_{i=1}^\infty U'$
of real $\zz_p$-inner product spaces $U'$, such that:
\begin{enumerate}
\item $\rr \subset U'$  (a canonical choice of the trivial representation)
\item $U$ is topologised as the union of all finite-dimensional $\zz_p$-subspaces of $U$
(each with the norm topology)
\item the $\zz_p$-action on each finite dimensional $G$-subspaces $V$ of
$U$ factors through a compact Lie group.
\end{enumerate}
Such an object $U$ is said to be \textbf{complete} if
every finite dimensional irreducible representation is contained (up to isomorphism) within $U$.
\end{definition}

A complete universe always exists, one can be obtained by taking all irreducible representations and taking the sum of each of these countably infinitely many times. This definition of a universe is \cite[Definition 3.1]{fauskeqpro} and below we give \cite[Definition 4.1]{fauskeqpro}.

\begin{definition}
A \textbf{$\pi_*$-equivalence} of orthogonal $\zz_p$-spectra is a map $f$ such that
$\pi_n^H(f)$ is an isomorphism for each open subgroup $H$.
\end{definition}

By \cite[Theorem 4.4]{fauskeqpro} we have a model structure on
$\zz_p$-spectra which behaves well with respect to the smash product of spectra.

\begin{theorem}
Let $U$ be a complete $\zz_p$-universe, then the category of orthogonal $\zz_p$-spectra
$\zz_p \Sp$ is a compactly generated, proper, monoidal model category satisfying the monoid axiom.
\end{theorem}

The model category $\zz_p \Sp$ is a $\zz_p \tscr$-model category, that is, it is enriched, tensored and cotensored
over $\zz_p \tscr$ in a manner compatible with the model structures, see \cite[Definition 4.2.18]{hov99}.

\begin{lemma}
The stable homotopy group functor, $\pi_n^{p^k \zz_p}$ is
co-represented by the spectrum $\Sigma^\infty (\zz_p/p^k)_+ \smashprod S^n$ for $n \geqslant 0$ and $\Sigma^\infty_{\rr^{-n}} (\zz_p/p^k)_+$ for negative $n$.
\end{lemma}
Thus the objects $\Sigma^\infty (\zz_p/p^k)_+$ (usually written as $\zz_p/p^k_+$) for $k  \geqslant 0$ are a set of generators for the homotopy category. This result is \cite[Lemma 4.6]{fauskeqpro} and the next is \cite[Proposition 7.10]{fauskeqpro}.

Segal-tom Dieck splitting also holds in this model category and is central to our computations.
\begin{proposition}\label{prop:splitting}
If $Y$ is a based $\zz_p$-space, then there is an isomorphism
of abelian groups
$$\bigoplus_{k \geqslant 0}
\pi_* \left( \Sigma^\infty \left(
(E \zz_p/p^k_+ \smashprod_{\zz_p/ p^k}
 Y^{p^k \zz_p}
\right)
\right)
\to \pi_*^{\zz_p}(\Sigma^\infty Y).
$$
\end{proposition}

The splitting result gives an additive description of $[S,S]^{\zz_p}_* \pi_*^{\zz_p} (S^0)$,
we also need a multiplicative description in degree zero, so we give
\cite[Lemma 7.11]{fauskeqpro}.

\begin{lemma}
The ring of self maps of the sphere spectrum in the homotopy category of $\zz_p$-spectra is
naturally isomorphic to
$\colim_{k \geqslant 0} [S,S]^{\zz/p^k}.$
\end{lemma}

One could repeat this entire construction using $p^n \zz_p$ as the group, choosing the set of subgroups to be those of form $p^k \zz_p$ for $k \geqslant n$, this would give a model category of $p^n \zz_p$-spectra. We mention this now as later we will need to move between this model category and the model category of $\zz_p$-spectra, using change of groups functors. The inclusion $i_n \co p^n \zz_p \to \zz_p$ induces a Quillen pair ($(\zz_p)_+ \smashprod_{p^n \zz_p} (-), i^*_n)$ between the model categories of $p^n \zz_p$-spectra and $\zz_p$-spectra (this also works in the space-level setting). We also have an inflation - fixed points Quillen adjunction between
$\zz_p$-spectra and $\zz/p^n$-spectra, coming from the projection $\sigma_n \co \zz_p \to \zz/p^n$.

Now we turn to the task of localising this model structure to obtain a
model category of rational $\zz_p$-spectra. We do so in a standard manner:
Bousfield localisation. Following \cite{barnesfinite} we introduce a
rational sphere spectrum $S^0 \qq$, designed so that the Bousfield localisation of
orthogonal $\zz_p$-spectra at this spectrum gives a model category whose
weak equivalences are precisely those maps which induce isomorphisms
on all rationalised homotopy groups.

The construction of $S^0 \qq$ is as follows, let $F = \oplus_{q \in \qq} \zz$ and let $R$ be the kernel of the map $F \to \qq$ which sends $1$ in factor $q$ to $q$. Thus we have a free resolution of $\qq$ as a $\zz$-module: $0 \to R \to F \to \qq$. We can realise the map $R \to F$ as a map of $\zz_p$-spectra: $\bigvee_R S^0 \to \bigvee_F S^0$ and take the cofibre which we call $S^0 \qq$.

\begin{theorem}
Let $U$ be a complete $\zz_p$-universe, then the model category of rational orthogonal $\zz_p$-spectra, $\zz_p \Sp_\qq$, is a compactly generated, proper, monoidal model category satisfying the monoid axiom.
\end{theorem}

Let us digress temporarily to the case of a general profinite group $G$. Repeating the above work we can make a model
category of rational $G$-spectra on a complete universe, which has generators of form $G/H_+$, where $H$ runs over the set of all open subgroups of $G$. Combining a calculation with a little more terminology we can prove a first attempt at classifying
rational $G$-spectra in terms of an algebraic model.

\begin{theorem}\label{thm:allinzero}
For $G$ a profinite group, the graded $\qq$-module $[G/H_+, G/K_+]^{G \Sp_\qq}_*$
is concentrated in degree zero.
\end{theorem}
\begin{proof}
To prove this we use the splitting result of Proposition \ref{prop:splitting}, via the isomorphisms:
$$[G/H_+,G/K_+]^G_*  \cong [S,G/K_+]^H_* \cong \pi_*^H(G/K_+)$$
where everything is stable and rational, of course. Now we split this homotopy group into a sum over the collection of open subgroups of $H$.
$$
\pi_*^H(G/K_+) \cong \oplus_{(L) \in {\textrm{Open}(H)}}
\pi_* \big( E W_H L_+ \smashprod_{W_H L} (G/K_+)^L  \big)
$$
Each $L$ in the above has finite index in $H$, so $W_H L$ is a finite group. The $G$-set $G/K$ is also finite, hence $(G/K)^L$ is a finite $W_H L$ set. Thus we can decompose $(G/K)^L$ as the following coproduct $\coprod_{i \in I} W_H L/M_i$, for $I$ some finite indexing set. We then have a series of isomorphisms which follow from the standard results: $X \smashprod_G G/H \cong X/H$, $(EG)/H = BH$ and $\pi_* (BH_+) \cong \qq$, for a finite group $H$.
$$
\begin{array}{rcl}
\pi_*\big( E W_H L_+ \smashprod_{W_H L} (G/K_+)^L  \big)
& \cong &
\oplus_{i \in I} \pi_*\big( E W_H L_+ \smashprod_{W_H L} (W_H L/M_i)_+ \big) \\
& \cong &
\oplus_{i \in I} \pi_*\big( (B M_i)_+ \big) \\
& \cong &
\oplus_{i \in I} \qq
\end{array}
$$
\end{proof}

\begin{definition}\label{def:orbitcategory}
For $G$ a profinite group, let $\ocal_G$ be the $\qq \leftmod$-enriched category with
objects the collection $G/H$, for $H$ open in $G$ and morphisms given by
$$\ocal_G (G/H, G/K) = [G/H_+, G/K_+]^{G \Sp_\qq}_*.$$
 A \textbf{right module over $\ocal_G$}
is a contravariant enriched functor from $\ocal_G$ to $\qq \leftmod$.
This category is also known as the category of \textbf{rational Mackey functors} for the group $G$.
\end{definition}
When $G$ is finite, this notion coincides with the usual concept
of Mackey functors for $G$. The category of differential graded
right modules over $\ocal_G$, $dg \rightmod \ocal_G$, has a model structure with weak equivalences and fibrations defined
objectwise.

\begin{theorem}\label{thm:ssapp}
For a profinite group $G$, the model category of rational $G$-spectra is Quillen equivalent to
the model category of differential graded right modules over $\rightmod \ocal_G$.
\end{theorem}
\begin{proof}
Since $\ocal_G$ is concentrated in degree zero,
\cite[Theorem A.1.1 and Proposition B.2.1]{ss03stabmodcat}
give the result.
\end{proof}

We now have our topological category and in one sense we have completed our task for a general profinite group $G$: we have an algebraic category that is equivalent to rational $G$-spectra. But it isn't always so clear what this algebraic model is, nor is it necessarily easy to construct objects in it. We would also like a method of calculating maps between two rational $G$-spectra via this category, so we look for an alternative description. Since $\ocal_G(G/H, G/K)$ is a module over the rational Burnside ring of $G$, our next step is to study that ring.

\newpage

\section{The Burnside ring of the \texorpdfstring{$p$}{p}-adic integers}\label{sec:burnside}

The Burnside ring $A(G)$ of a finite group $G$ is the Grothendieck ring of finite $G$-sets. There is an isomorphism between the Burnside ring and the self maps of the sphere spectrum in the homotopy category of $G$-spectra.

\begin{definition}
We define the Burnside ring of $\zz_p$ to be $\colim_n A(\zz/p^n)$,
the colimit over $n$ of the Burnside rings of the finite groups
$\zz/p^n$.
\end{definition}

It is easy to see that $\colim_n A(\zz/p^n)$ is isomorphic to
the Grothendieck ring of finite discrete $G$-sets:
those finite sets $X$ where $\zz_p$ acts on $X$ by first projecting to $\zz/p^n$
for some $n$. From here on we only work rationally,
so $A(\zz_p)$ or $A(\zz/p^n)$ will mean the rationalised Burnside rings.

For our calculations an alternative description of the Burnside rings is useful.
Let us consider a finite group $G$, associated to this is the discrete space of conjugacy classes of subgroups of $G$, $\scal G$. Then the Burnside ring of $G$ is isomorphic to $C (\scal G, \qq)$, the ring of continuous maps from $\scal G$ to the discrete topological space $\qq$.

A group homomorphism $G_i \to G_j$ induces a map $\scal G_i \to \scal G_j$, which induces a map
$C (\scal G_j, \qq) \to C (\scal G_i, \qq)$. Let's use this to find a description of the Burnside ring of $\zz_p$ which we know to be isomorphic to the colimit over $n$ of the rings $C (\scal \zz/p^n, \qq)$.

Some easy calculations give the following diagram, where we have added the undefined term $\scal \zz_p$:
$$\xymatrix{
  {\scal \zz/p^0}
& {\scal \zz/p^1} \ar[l]
& {\scal \zz/p^2} \ar[l]
& {\scal \zz/p^3} \ar[l] & \ar[l] \cdots &\ar[l] \scal \zz_p \\
  {\zz/p^0}
& {\zz/p^1} \ar[l]
& {\zz/p^2} \ar[l]
& {\zz/p^3} \ar[l] & \ar[l] \cdots &\ar[l] \zz_p\\
& {p\zz/p^1} \ar[ul]
& {p\zz/p^2} \ar[l]
& {p\zz/p^3} \ar[l] & \ar[l] \cdots &\ar[l] p\zz_p\\
&& {p^2\zz/p^2} \ar[ul]
& {p^2\zz/p^3} \ar[l] & \ar[l]  \cdots &\ar[l] p^2\zz_p\\
&&& {p^3\zz/p^3} \ar[ul]& \ar[l] \cdots &\ar[l] p^3 \zz_p\\
&&&& \ar[ul] \cdots
}$$

To explain this concretely, consider $p=2$, then the above diagram says that the
$\zz/8$-subgroups $4\zz / 8$ (a two element group) and $8\zz / 8$ (the trivial group)
are identified under the  projection $\zz/8 \to \zz/4$. Whereas the subgroup
$2\zz/8$ (the cyclic group of order 4) becomes $2\zz/4$ under the projection.

The set $\scal \zz_p$ of closed subgroups of $\zz_p$ can be given a topology via the Hausdorff metric on $\zz_p$.
We can describe $\scal \zz_p$ as the subspace of $\rr$ consisting of those points $p^{-n}$ for $n \geqslant 0$
and the point $0$. We label the point $p^{-n}$ as $p^n \zz_p$ and the point $0$
as the trivial group. Define the function $e_n \co \scal {\zz_p} \to \qq$ to be that map which sends $p^{-n}$ to 1 and all other points to zero.

Note that $C (\scal \zz_p, \qq)$ is  the ring of eventually constant rational sequences: for large enough $k$ all points $p^k \zz_p$ are sent to the same rational number. We can write such a sequence as $(a_0, a_1, a_2, \dots)$ where $a_n$ is the value of the sequence at $p^n \zz_p$.

\begin{lemma}
There is an isomorphism
$\colim_n C (\scal \zz/p^n, \qq) \to C (\scal \zz_p, \qq)$
induced by the maps $\scal \zz_p \to \scal \zz/p^n$,
which sends $p^k \zz_p$ to $p^k \zz/p^n$
and the trivial group to $p^n \zz/p^n$.
\end{lemma}
\begin{proof}
Take some element of $\colim_n C (\scal \zz/p^n, \qq)$, then
it is represented by some element $f \in C (\scal \zz/p^n, \qq)$, for
suitably large $n$.
Now we define $g \co \scal \zz_p \to \qq$ from $f$,
let $g (p^k \zz_p) = f(p^k \zz/p^n)$ for $0 \leqslant k \leqslant n$
and let $g (p^k \zz_p) = f(p^n \zz/p^n)$ for $k \geqslant n$.
The map $g$ is precisely that map induced from $f$ by $\scal \zz_p \to \scal \zz/ p^n$.
This new map is continuous as it is eventually constant. It is easy to
check that this assignment is well-defined and gives an isomorphism.
\end{proof}

The ring $C (\scal \zz_p, \qq)$ is part of a pullback square in the category of
commutative rings:
$$\xymatrix{
C (\scal \zz_p, \qq) \ar[r] \ar[d] & \qq \ar[d] \\
\prod_{n \geqslant 0} \qq \ar[r] & \colim_k \prod_{n \geqslant k} \qq
}$$
We will need this kind of pullback square later and we also note that this ring
occurs in the study of rational $O(2)$-spectra, we discuss this further in
section \ref{sec:relations}.

It is also important to understand how the inclusion map $p^n \zz_p \to \zz_p$
induces a map of Burnside rings $i_n^* \co A(\zz_p) \to A(p^n \zz_p)$.
Since $p^n \zz_p$ and $\zz_p$ are isomorphic as abelian groups
($\zz_p \to p^n \zz_p$, $x \mapsto p^n x$ is an isomorphism),
$A(p^n \zz_p)$ is isomorphic as a ring to to $A(\zz_p)$, but
$i_n^*$ is \emph{not} an isomorphism. We can describe $A(p^n \zz_p)$ as the ring of
eventually constant sequences $(a_n, a_{n+1}, a_{n+2}, \dots)$ where
$a_k$ is the value of the sequence at $p^k \zz_p$, for $k \geqslant n$. This gives a nice definition of $i_n^*$.

\begin{lemma}
The map $i_n^* \co A(\zz_p) \to A(p^n \zz_p)$
acts by truncation, it takes the eventually constant
sequence $(a_0, a_1, a_2, \dots)$ to
$(a_n, a_{n+1}, a_{n+2}, \dots)$.
\end{lemma}
\begin{proof}
This is easy to see, the subgroups of $p^n \zz_p$
have form $p^k \zz_p$ for $k \geqslant n$, these are sent to
$p^k \zz_p$ as a subgroup of $\zz_p$.
\end{proof}

Now we describe the algebraic model that will represent the homotopy category of  rational $\zz_p$-spectra. By represent, we mean that we will produce (in section \ref{sec:proof}) a series of Quillen equivalences relating the model category of rational $\zz_p$-spectra from section \ref{sec:ratspectra} and the model category of differential graded objects in the algebraic model.

\section{The Algebraic Model}\label{sec:model}
Now we can describe $\acal = \acal(\zz_p)$, the algebraic model for rational $\zz_p$-spectra. We begin with a brief introduction to sets with an action of $\zz_p$.

\begin{definition}
A $\qq$-module $M$ is said to be a \textbf{discrete $\zz_p$-module} if
there is a group action $\zz_p \times M \to M$ such that for any $m$,
the action of $\zz_p$ on $m$ factors through some finite quotient of
$\zz_p$.
\end{definition}
The purpose of this definition is that a $\zz_p$-module $M$ is discrete if and only if
the action map $\zz_p \times M \to M$ is a continuous map, using the discrete
topology on $M$ and the standard topology on $\zz_p$.

\begin{lemma}
If $M$ is a set with a not necessarily continuous action of $\zz_p$,
then there is a natural $\zz_p$-equivariant inclusion map of a discrete $\zz_p$-set into $M$:
$$
\colim_n M^{p^n \zz_p} \to M.
$$
This map is an isomorphism of and only if $M$ is discrete.
\end{lemma}
We can think of $\colim_n M^{p^n \zz_p}$ as the discrete part of $M$, if
$N \to M$ is a $\zz_p$-equivariant map, with $N$ discrete, then this map factors
$N \to \colim_n M^{p^n \zz_p} \to  M$.

The category of discrete $\zz_p$-modules is a closed monoidal category,
with tensor product given by tensoring over $\qq$ and homomorphism object given by taking the discrete part of $\hom_\qq(-,-)$, which is a $\zz_p$-module by conjugation.

\begin{proposition}
There is a cofibrantly generated model structure on the category of differential graded rational discrete $\zz_p$-modules, where a map is a weak equivalence if and only if it is a homology isomorphism and the fibrations are the surjections.  Let $\qq[\zz_p]_d \leftmod$ denote this model category of discrete $\zz_p$-modules.
\end{proposition}
\begin{proof}
A colimit of discrete modules is discrete, limits are defined by taking the limit in the category of $\qq$-modules and then taking the discrete part.
The generating cofibrations are the maps of the form $\qq[\zz_p/p^n] \otimes i$, where $i$ is a generating cofibration for the model structure on differential graded $\qq$-modules. Similarly, the generating acyclic cofibrations are the maps of the form $\qq[\zz_p/p^n] \otimes j$, where $j$ is a generating acyclic cofibration for the model structure on differential graded $\qq$-modules.

Focusing on these generating sets, one sees that a map $f$ is a fibration if and only if each $f^{p^n\zz_p}$ is a surjection and a $f$ is a weak equivalence if and only if each $f^{p^n\zz_p}$ is a homology isomorphism. Since the modules are discrete and rational, the fibrations are precisely the surjections and the weak equivalences are precisely the homology isomorphisms.
\end{proof}

Most of the above results on $\zz_p$-modules hold in the case of a general profinite group $G$.

\begin{definition}\label{def:algmodel}
An object $M$ of the category $\acal$ is a collection of $\qq[\zz_p/p^k]$-modules $M_k$ for $k \geqslant 0$ and a discrete $\zz_p$-module $M_\infty$, with a specified map of $\zz_p$-modules $M_\infty \to  \colim_n \prod_{k \geqslant n} M_k$, called the \textbf{structure map} of $M$.
A map $f \co M \to N$ in this category is then a collection of $\qq[\zz_p/p^k]$-module maps $f_k \co M_k \to N_k$ and a map of $\zz_p$-modules $f_\infty \co M_\infty \to N_\infty$ which give a commutative square:
$$
\xymatrix{
M_\infty  \ar[r] \ar[d] &
\colim_n \prod_{k \geqslant n} M_k \ar[d] \\
N_\infty  \ar[r]        &
\colim_n \prod_{k \geqslant n} N_k .
}
$$
We call this category the  \textbf{algebraic model for rational $\zz_p$-spectra}.
\end{definition}

Note that for an object $M$ of $\acal$, the $\zz_p$-module $\colim_n \prod_{k \geqslant n} M_k$ is not, in general, discrete. However, the image of the structure map of $M$, which is a submodule of $\colim_n \prod_{k \geqslant n} M_k$, is discrete.

We introduce a number of special objects of $\acal$, these will be used to create the model structure on $\acal$ and will be needed for the main proof.

\begin{definition}\label{def:specialAobjects}
For $n \geqslant 0$, let $A(n)$ be the object which takes value $\qq[\zz_p/p^n]$ in all entries greater than or equal to $n$ (including infinity) and is zero elsewhere with structure map given by the diagonal (we will also use $U$ for $A(0)$). Let $E(n)$ be that object which takes value $\qq[\zz_p/p^n]$ at $n$ and is zero everywhere else and let $L(n)$ be that object which takes value $\qq[\zz_p/p^n]$ at $\infty$ and is zero everywhere else.
\end{definition}

As we will see later, once we have our Adams spectral sequence,  $A(n)$ is the algebraic model for the spectrum $\zz_p/p^n_+$ and the collection of the $A(n)$ form a generating set for the algebraic model. The object $U$ will be seen to be the unit of the monoidal product on $\acal$.  The collections $L(n)$ and $E(n)$ will also form a set of generators for the category, these are sometimes easier to work with, as there are fewer maps between these generators. In terms of spectra, $E(n)$ represents $e_n \zz_p/p^n$ and $L(n)$ is the algebraic model for $\hocolim_k f_k \zz_p/p^n$, as we show in lemma \ref{lem:represents}.

Let $k,m,n \geqslant 0$ with $k \neq n$:
$$\begin{array}{lcl}
\hom_{\acal}(U,U) & = & A(\zz_p) \\
\hom_{\acal}(U,L^n) & = & \qq    \\
\hom_{\acal}(U,E^n) & = & \qq    \\
\hom_{\acal}(L^n,U) & = & 0    \\
\hom_{\acal}(L^n,L^n) & = & \qq[\zz_p/p^n]    \\
\hom_{\acal}(L^n,E^m) & = & 0    \\
\hom_{\acal}(L^m,L^n) & = & \qq[\zz_p/p^{\textrm{min}(m,n)}]   \\
\hom_{\acal}(E^n,U) & = & \qq  \\
\hom_{\acal}(E^n,E^n) & = & \qq[\zz_p/p^n]     \\
\hom_{\acal}(E^n,L^m) & = & 0    \\
\hom_{\acal}(E^n,E^m) & = & 0.
\end{array}$$

\begin{definition}
For an object $M \in \acal$, let
$\seq_n(M) = \prod_{k \geqslant n} M_k$ and let
$\tails(M) = \colim_n \prod_{k \geqslant n} M_k$.
Let $\bignplus_n M$ be the following pullback.
$$
\xymatrix{
\bignplus_n M \ar[r] \ar[d] &
\seq_n M \ar[d] \\
M_\infty  \ar[r]        &
\tails(M).
}
$$
We call $\bignplus_n M$ the set of \textbf{eventually specified sequences in $M$, starting at $n$}.
\end{definition}

The map $M \to \tails(M)$ can be thought of giving a set of permissible behaviours for sequences of elements
in the modules $M_n$. An element of $\bignplus$ is then an element $m$ and a sequence $(m_k)$, with $m_k \in M_k$ and $k \geqslant n$ such that the sequence agrees with the image of $m$ in $\tails(M)$, in other words, the sequence
$(m_k)$ is eventually specified by $m$. The notation $\bignplus$ is meant to look like some combination of $\oplus$ and $\prod$, since the elements of $\bignplus_n M$ are infinite sequences where we still have some control over their eventual behaviour.  Consider a special class of objects of $\acal$, those where $M_\infty = M_k$ for all $k$ and the structure map is the diagonal. Then an eventually specified sequence in $M$ is just an eventually constant sequence. We extend all these definitions to objects of $dg \acal$ without further decoration.

We introduce another category, related to $\acal$ by an adjunction that is useful in understanding the monoidal structure of $\acal$.

\begin{definition}
Let $\bcal$ be the category of sheaves of discrete $\zz_p$-modules on the space $\scal \zz_p$. We call this the \textbf{sheaf model for rational $\zz_p$-spectra}.
\end{definition}
Given any object $A \in \acal$, one defines a sheaf $UA$ on $\scal \zz_p$ by setting $UA(S_n) = \bignplus_n A$, where $S_n$ is the open set of $\scal \zz_p$ consisting of all points of the form $p^k \zz_p$ for $k \geqslant n$ and the trivial group. At the singleton set of $p^n \zz_p$ we let $UA$ take value $A_n$.
Given any $M \in \bcal$, define $RM$ to be that object of $\acal$ which takes value $M(p^n \zz_p)^{p^n \zz_p}$ at $n$ and at infinity is given by following pullback:
$$
\xymatrix{
(RM)_\infty \ar[r] \ar[d] &
\colim_n \prod_{k \geqslant n} M(p^n \zz_p)^{p^n \zz_p} \ar[d]\\
M_e \ar[r] &
\colim_n \prod_{k \geqslant n} M(p^n \zz_p).
}
$$
In the above, $M_e$ is the stalk at the trivial group and the lower horizontal map is given by realising stalks as an element of $\colim_n M(S_n)$ and then projecting onto the product of the singleton sets $p^k \zz_p$ for $k \geqslant n$.

\begin{lemma}
The functors $U$ and $R$ form an adjunction
$$
U : \acal
\overrightarrow{\longleftarrow}
\bcal : R
$$
where the left adjoint $U$ is full and faithful.
\end{lemma}

\begin{theorem}
The category $dg \acal(\zz_p)$ is a proper, cofibrantly generated model category. The fibrations are those maps $f$ which are a surjection at each $n$ and infinity. The weak equivalences are those maps $f$ which are homology isomorphisms at each $n$ and infinity.
\end{theorem}
\begin{proof}
Given a small diagram $M^i$ of objects of $\acal$, the colimit is formed
by taking the colimit at each point, so $(\colim_i M^i)_k = \colim_i (M^i_k)$
and similarly so at infinity.  The map below induces (via the universal properties of colimits)
a structure map for $\colim_i M^i$.
$$
M^i_\infty \longrightarrow
\colim_n \prod_{k \geqslant n} M_k^i \longrightarrow
\colim_n \prod_{k \geqslant n} \colim_i M_k^i .
$$

Limits are harder to define, as before $(\lim_i M^i)_k = \lim_i (M^i_k)$,
but at infinity we must do the following. Form the pullback diagram
$$\xymatrix{
P
\ar[r]
\ar[d] &
\colim_N \prod_{k \geqslant N} \lim_i M_k^i
\ar[d] \\
\lim_i M_\infty^i
\ar[r] & \lim_i \colim_N \prod_{k \geqslant N} M_k^i
}$$
then let $(\lim_i M^i)_\infty$ be the discrete submodule of $P$. It is routine to check that these constructions of colimits and limits have the appropriate universal properties. One can also define the limit of the $M^i$ as $R\lim_i (UM^i)$, $R$ applied to the limit of the associated objects of $\bcal$.

The generating cofibrations and acyclic cofibrations are the sets
$$I_\acal = \{ A(n) \otimes I \}$$
$$J_\acal = \{ A(n) \otimes J \}$$
where $I$ ($J$) is the set of generating cofibration (acyclic cofibrations) for
$dg \qq \leftmod$. To prove that this gives the model structure we have described, we must identify the
fibrations and acyclic fibrations.

A map $f$ has the right lifting property with respect to $I$ (or $J$) if and only if $(\bignplus_n f)^{p^n \zz_p}$
is a surjection for each $n$. It is routine to show that $f$ has the lifting property if and only if each $f_n$
and $f_\infty$ is a surjection (or surjection and homology isomorphism). The main points needed are
the natural isomorphism
$$(\bignplus_n M)^{p^n \zz_p} = (\bignplus_{n+1} M)^{p^n \zz_p} \oplus M_n$$
and the general technique of first dealing with the term from $M_\infty$ and
then dealing with the remaining finite number of terms in the $M_k$ which are not specified by $\sigma M_\infty$.

Left properness is easy to see, for right properness we use the fact that for any finite limit diagram $M^i$, $(\lim_i M^i)_\infty = \lim_i (M^i_\infty)$.
\end{proof}

\begin{lemma}
The pair $(U,R)$ are a Quillen pair when $dg \acal$ is given the above model structure and $dg \bcal$ is given the levelwise model structure.
The pair $(U,R)$ are a Quillen equivalence when $dg \acal$ is given the above model structure and $dg \bcal$ is given the model structure determined by applying $U$ to the generating sets of cofibrations and acyclic cofibrations of $dg \acal$.
\end{lemma}
This lemma is an example of what is sometimes known as the cellularisation principle, where we co-localise the model structure on $dg \bcal$ at the cells $U A(n)$, for $n \geqslant 0$.

Now we can use the adjunction $(U,R)$ to define a monoidal product and internal homomorphism object on $\acal$. We start with the category of rational discrete $\zz_p$-modules, which is a closed monoidal category with monoidal product $\otimes_\qq$. This monoidal product induces a tensor product on $\bcal$, defined by $(M \otimes N)(U) = M(U) \otimes_\qq N(U)$. The internal homomorphism object is as usual for sheaves, $\hom(M,N)(U) = \hom_{\bcal_{|U}}(M_{|U}, N_{|U})$, where the right hand side is the discrete part of the set of maps of rational sheaves from $M$ restricted to $U$ to $N$ restricted to $U$ (where $\zz_p$ acts by conjugation).

Now let $M$ and $N$  be in $\acal$, then the monoidal product of $M$ and $N$ is $M \otimes N$, which at $n$ takes value $M_n \otimes_\qq M_n$ and at $\infty$ takes value $M_\infty \otimes N_\infty$.  The structure map is as follows.
$$
\begin{array}{rcll}
M_\infty \otimes N_\infty
& \rightarrow &
\colim_n \prod_{k \geqslant n} M_k \otimes
\colim_l \prod_{l \geqslant m} N_l \\
& \rightarrow &
  \colim_{n,l} (\prod_{k \geqslant n} M_k \otimes
   \prod_{l \geqslant m} N_l ) \\
& \overset{\cong}{\leftarrow} &
  \colim_{n} (\prod_{k \geqslant n} M_k \otimes
   \prod_{l \geqslant n} N_k ) \\
& \rightarrow &
  \colim_{n} (\prod_{k \geqslant n} M_k \otimes N_k ) \\
\end{array}
$$

The internal function object requires us to briefly use $\bcal$, we define $\hom(A,B) = R \hom(UA, UB)$, the functor $R$ applied to the $\bcal$-homomorphism object of maps from $UA$ to $UB$. The following series of adjunctions shows that this object has the correct properties to be the internal homomorphism object of $\acal$:
$$
\acal(A \otimes B, C)
\cong \bcal (UA \otimes UB, UC)
\cong \bcal (UA , \hom(UB, UC)
\cong \acal (A , R\hom(UB, UC)
$$

\begin{lemma}
The model category $dg \acal$ is a monoidal model category that satisfies the monoid axiom.
\end{lemma}
\begin{proof}
The proof essentially relies on the corresponding statements for $\qq$-modules and the fact that $\qq[\zz_p/p^n] \otimes_\qq \qq[\zz_p/p^m]$ is isomorphic, as a $\zz_p$-module, to the direct sum of $p^m$-copies of $\qq[\zz_p/p^n]$, for $n \geqslant m$.
\end{proof}

\begin{proposition}
The model category $dg \acal$ is a $dg \qq \leftmod$-model category via the
symmetric monoidal adjunction
$$\begin{array}{rcl}
U \otimes_{\qq} (-) : dg \qq \leftmod
& \overrightarrow{\longleftarrow} &
dg \acal : \hom_{\acal}(U,-).
\end{array}
$$
\end{proposition}

\newpage

\begin{proof}
The left adjoint is clearly strong symmetric monoidal, we also note that the right adjoint, applied to an object $M \in \acal$, can be described as the following pullback.
$$
\xymatrix{
\hom_\acal(U,M) \ar[r] \ar[d] & \prod_{k \geqslant 0} M_k^{\zz_p/p^k} \ar[d] \\
M_\infty^{\zz_p} \ar[r] & \colim_N \prod_{k \geqslant N} M_k
}
$$
\end{proof}

We now show that our algebraic category is really just the category of rational Mackey functors. The main advantage of using $\acal$ is the ease of description and construction, especially when we construct the Adams short exact sequence.

\section{Rational Mackey Functors}\label{sec:mackey}

To understand rational $\zz_p$-Mackey functors, it is important to have a good description of $\ocal_{\zz_p}$, see definition \ref{def:orbitcategory}, we can calculate this via tom-Dieck splitting, but to obtain a more useful and neater characterisation, we need to see how the Burnside ring appears in $\ocal_{\zz_p}$.

The set of $\zz_p$-equivariant maps of spaces from $\zz_p/p^n$ to $\zz_p/p^m$ is isomorphic to $\zz_p/p^{min(n,m)}$. Any such map $f$ determines a map of suspension spectra, indeed, there is an injective map
$$\phi \co \map(\zz_p/p^n,\zz_p/p^m)^{\zz_p} \to [\zz_p/p^n_+, \zz_p/p^m_+]^{\zz_p}.$$
 Recall that there is an isomorphism
$A(p^{\max(n,m)}\zz_p) \to [S^0, S^0]^{p^{\max(n,m)}\zz_p}$
and then an injective map $\psi \co [S^0, S^0]^{p^{\max(n,m)}\zz_p} \to [S^0, S^0]^{p^n \zz_p}$
given by sending some eventually constant sequence starting at the maximum of $n$ and $m$
to one starting at $n$ by filling in the extra terms at the beginning with zeros.

Choose some map $f \in \map(\zz_p/p^n,\zz_p/p^m)^{\zz_p}$ and some $x \in A(p^{\max(n,m)}\zz_p)$, then we have a map
$$\phi(f) \smashprod \psi(x) \in [S^0, \zz_p/p^m_+]^{p^n \zz_p}
\cong [\zz_p/p^n_+, \zz_p/p^m_+]^{\zz_p}.$$
The projection $[S^0, \zz_p/p^m_+]^{p^n \zz_p} \to [S^0, S^0]^{p^n \zz_p}$
takes $\phi(f) \smashprod \psi(x)$ to $\psi(x)$. Thus for each $f$, the association $x \mapsto \phi(f) \smashprod \psi(x)$ is injective. It is also clear that
$\phi(f) \smashprod \psi(x)$ and $\phi(g) \smashprod \psi(x)$ are not homotopic whenever
$f \neq g \in \map(\zz_p/p^n,\zz_p/p^m)^{\zz_p}$.

We have thus defined a map
$$
\chi \co A(p^{\max(n,m)}\zz_p) \times \map(\zz_p/p^n, \zz_p/p^m)^{\zz_p}
\longrightarrow
[\zz_p/p^n_+, \zz_p/p^m_+]^{\zz_p}
$$
and we have seen that it is injective.

\begin{proposition}\label{prop:orbcatcalc}
For any $n ,  m \geqslant 0$, the map below is an isomorphism
$$
\chi \co A(p^{\max(n,m)}\zz_p) \otimes \qq[\zz_p/p^{min(n,m)}]
\longrightarrow
[\zz_p/p^n_+, \zz_p/p^m_+]^{\zz_p}
$$
\end{proposition}
\begin{proof}
We know that the map $\chi$ exists and is injective, a routine calculation of the target by tom-Dieck splitting shows that it is also surjective. The only point worthy of note there is that when written in terms of rational Grothendieck rings, the  idempotent $e_{n}$ is given by $p^{-n} \zz/p^n - p^{-n-1} \zz/p^{n+1}$. With this understood, we can see where a term such as $e_n \otimes x$ is sent to in the tom-Dieck splitting of the target.
\end{proof}

We briefly digress to consider the case of a general ring, $R$. Assume that there is a countable collection of idempotents $\{e_k| k \geqslant 0\}$, where $e_i e_j = 0$ for $i \neq j$. We do not require that these idempotents be non-zero, hence a finite collection of orthogonal idempotents gives such a collection, by setting $e_k = 0$ for sufficiently large $k$. Define a new collection of idempotents by $f_k = 1- \Sigma_{i=0}^{k-1}e_i$.

The ring $R$ can be described as the pullback in the category of rings of the diagram
$$
\xymatrix{
R  \ar[r] \ar[d] &
\prod_{k \geqslant 0} e_k R\ar[d] \\
\colim_n f_n R  \ar[r]        &
\colim_n \prod_{k \geqslant n} e_k R.
}
$$
This pullback description also applies with $R$ in the above replaced by any $R$-module $M$.
In the case of a finite collection, $\colim_n \prod_{k \geqslant n} e_k R=0$
so this pull back is simply describing $R$ as a product of its idempotently split pieces.
In this case the $e_i$ sum to the unit if and only if $\colim_n f_n R =0$.

We want to understand in more detail what information is needed to define a $\zz_p$-Mackey functor. We are guided by the standard example: the homotopy group Mackey functor of a spectrum $X$. This takes the object $\zz_p/ p^k \zz_p$ of $\ocal_{\zz_p}$ to the graded $\qq$-module $\pi_*^{p^k \zz_p}(X)$. We can also describe this homotopy group Mackey functor as the functor below, which has codomain the category of graded $\qq$-modules:
$$[-,X]^{\zz_p}_* \co \ocal_{\zz_p} \to g \qq \leftmod.$$

Let $M$ be a rational Mackey functor, then for each $k \geqslant 0$, one has a $\qq$-module $M(\zz_p / p^k \zz^p)$, the spectrum $\zz_p / p^k \zz^p_+$ has a right action by $\zz_p/p^k$, so this $\qq$-module also has an action $\zz_p/p^k$. Further $M(\zz_p / p^k \zz^p)$ is a module over $A(p^k \zz_p)$, so we have a pullback diagram:
$$
\xymatrix{
M(\zz_p/p^k \zz_p) \ar[r] \ar[d] &
\prod_{n \geqslant k} e_n M(\zz_p/p^k \zz_p) \ar[d] \\
\colim_n f_n M(\zz_p/p^k \zz_p)  \ar[r]        &
\colim_n \prod_{m \geqslant n} e_m M(\zz_p/p^k \zz_p).
}
$$
Thus $M$ can be reconstructed from the pieces $e_n M(\zz_p/p^k \zz_p) $
and the collection of projection maps $f_n M(\zz_p/p^k \zz_p) \to \prod_{m \geqslant n} e_m M(\zz_p/p^k \zz_p)$ for varying $n$ and $k$.
Following the decomposition of rational Mackey functors in \cite{gre98b}, we expect that this decomposition can be simplified. In particular, the pieces $e_n M(\zz_p/p^k \zz_p)$ should be recoverable from the terms $M_n = e_n M(\zz_p / p^n)$. Instead of needing all the projection maps we would expect to only need a module $M_\infty = \colim_n M(p^n \zz_p)$ and a structure map $M_\infty \to \tails(M)$.

\begin{theorem}
Let $M$ be a rational Mackey functor for $\zz_p$ and define $f_k = 1- \Sigma_{i=0}^{k-1} e_i$. Then, for $k \leqslant n$, there are isomorphisms (natural in $M$):
$$
\begin{array}{lclcl}
e_n M(\zz_p / p^k) & \cong &
(e_n M(\zz_p / p^n))^{p^k \zz_p / p^n } &=& M_n^{p^k \zz_p / p^n} \\
\colim_a (f_a M(\zz_p /p^k )) & \cong &
(\colim_a M(p^a \zz_p))^{p^k \zz_p} &=& M_\infty^{p^k \zz_p} .
\end{array}
$$
Hence we have the formula below for recovering a Mackey functor from the split pieces and the structure map.
$$
\xymatrix{
M(\zz_p /p^k \zz_p)  \ar[r] \ar[d] &
\prod_{m \geqslant k} M_m^{p^k \zz_p / p^m \zz_p} \ar[d] \\
M_\infty^{p^k \zz_p}  \ar[r]        &
\colim_n \prod_{m \geqslant n} M_m^{p^k \zz_p / p^m \zz_p}
}
$$
\end{theorem}

\begin{proof}
A variation on the Yoneda lemma states that if $F$ is a contravariant $\ccal$-enriched functor from a $\ccal$-enriched category $\ecal$ to $\ccal$, the canonical map $\int^{a\in \ecal} \ccal(x,a) \otimes F(a) \to F(x)$ is a natural isomorphism. In our setting this says that there is a natural isomorphism
$$\int^{\zz_p/p^m} [\zz_p/p^n, \zz_p/p^m]^{\zz_p} \otimes_\qq M(\zz_p/p^m) \to M(\zz_p/p^n).$$

For each element $x$ of $\zz_p/p^n$, we have a right multiplication map $r_x \co \zz_p/p^n \to \zz_p/p^n$. From these we construct $\text{Av}_{n,k} = p^{n-k} \Sigma_{x \in p^k \zz_p/p^n} r_x$, a map in the homotopy category of rational $\zz_p$-spectra. This map induces an idempotent map
$$(\text{Av}_{n,k})^* \co
[\zz_p/p^n, \zz_p/p^m]^{\zz_p} \to
[\zz_p/p^n, \zz_p/p^m]^{\zz_p}.$$
The image of this map is $\left\{ [\zz_p/p^n, \zz_p/p^m]^{\zz_p} \right\}^{p^k\zz_p}$, that is, the algebraic fixed points of this discrete $\zz_p$-module. Since fixed points in $\qq$-modules correspond to orbits, we can also write this as $\left\{ [\zz_p/p^n, \zz_p/p^m]^{\zz_p} \right\} / {p^k\zz_p}$. By naturality we obtain an isomorphism
$$\int^{\zz_p/p^m} \left\{  [\zz_p/p^n, \zz_p/p^m]^{\zz_p} \right\}/{p^k\zz_p} \otimes_\qq M(\zz_p/p^m)
\to M(\zz_p/p^n)/{p^k\zz_p} \cong M(\zz_p/p^n)^{p^k\zz_p}.$$

A similar argument shows that there are natural isomorphisms as below.
$$
\begin{array}{rcl}
\int^{\zz_p/p^m}
e_n[\zz_p/p^k, \zz_p/p^m]^{\zz_p}  \otimes_\qq M(\zz_p/p^m)
& \cong & e_n M(\zz_p/p^k) \\
\int^{\zz_p/p^m} \left\{  e_n[\zz_p/p^n, \zz_p/p^m]^{\zz_p} \right\}^{p^k\zz_p} \otimes_\qq M(\zz_p/p^m)
& \cong & e_n M(\zz_p/p^n)^{p^k\zz_p} \\
\end{array}
$$
We want to prove that the right hand sides are isomorphic, we do so by proving that the left hand sides are isomorphic, this reduces to showing that the $\qq$-modules  $\left\{  e_n[\zz_p/p^n, \zz_p/p^m]^{\zz_p} \right\}/{p^k\zz_p}$ and
$e_n[\zz_p/p^k, \zz_p/p^m]^{\zz_p}$ are naturally isomorphic for all $m \geqslant 0$.
This amounts to proving that $e_n  \text{Av}_{n,k} \zz_p/p^n$ and $e_n  \zz_p/p^k$
are weakly equivalent. We delay this to attend to the second case for a short while.

For the second statement we follow the same pattern as before, so we need a pair of natural isomorphisms:
$$
\begin{array}{rcl}
\int^{\zz_p/p^m} \colim_a f_a[\zz_p /p^k, \zz_p/p^m]^{\zz_p}
\otimes_\qq M(\zz_p/p^m)
& \cong & \colim_a (f_a M(\zz_p /p^k \zz_p)) \\
\int^{\zz_p/p^m}       \left\{    \colim_a [\zz_p/p^a, \zz_p/p^m]^{\zz_p}  \right\}/{p^k\zz_p}
\otimes_\qq M(\zz_p/p^m)
& \cong & (\colim_a M(p^n \zz_p))^{p^k \zz_p}. \\
\end{array}
$$
This would follow from the existence of a natural isomorphism between the $\qq$-modules $\colim_a f_a[\zz_p /p^k, \zz_p/p^m]^{\zz_p}$
and $\left\{    \colim_a [\zz_p/p^a, \zz_p/p^m]^{\zz_p}  \right\}/{p^k\zz_p}$.
We may as well assume that $a$ is larger than $k$, then we have the following collection of maps
$$
\begin{array}{rcl}
f_a [\zz_p /p^k, \zz_p/p^m]^{\zz_p}
& \to &
[\zz_p /p^k, \zz_p/p^m]^{\zz_p} \\
& \to &
[\zz_p /p^a, \zz_p/p^m]^{\zz_p} \\
& \to &
\left\{[\zz_p/p^a, \zz_p/p^m]^{\zz_p}  \right\}/{p^k\zz_p}
.
\end{array}
$$
Which are, from top to bottom, the inclusion of an idempotent summand,
the map induced by $\zz_p/p^a \to \zz_p/p^k$ and the quotient map.
The composite of these maps is an isomorphism if $\text{Av}_{a,k} \zz_p/p^a$ and $f_a  \zz_p/p^k$ are weakly equivalent (with this weak equivalence being natural in $a$).

Now if $\text{Av}_{a,k} \zz_p/p^a$ and $f_a \zz_p/p^k$ are weakly equivalent for $a \geqslant k$, then $e_a \text{Av}_{a,k} \zz_p/p^a$ and $e_a f_a  \zz_p/p^k \simeq e_a \zz_p/p^k$ are also weakly equivalent, which is our unfinished business from the first statement.

Consider the projection map $\zz_p/p^a \to \zz_p/p^k$, since $\zz_p/p^k$ is fixed under the action of $p^k \zz_p$, there is a map (in the homotopy category) from
$\text{Av}_{a,k} \zz_p/p^a$ to $\zz_p/p^k$. Now compose with the map $\zz_p/p^k \to f_a \zz_p/p^k$ to obtain our map $\text{Av}_{a,k} \zz_p/p^a \to f_a \zz_p/p^k$. Our calculations of $\ocal_{\zz_p}$ in proposition \ref{prop:orbcatcalc} show that we have our weak equivalence.

Now we are just required to show that one can build $M$ from the pieces $M_m$ and the map $M_\infty \to \tails (M)$, but this follows immediately from the two isomorphisms of the theorem and our earlier decomposition of $M(\zz_p/p^k \zz_p)$.
\end{proof}

Now we know precisely what information is needed to define a rational Mackey functor for $\zz_p$, as expected, this information is sufficient to make an object of $\acal$.

\begin{lemma}
Any Mackey functor $M$ defines an object of $\acal$.
\end{lemma}
\begin{proof}
For each $k \geqslant n$ we have a map $M(\zz_p /p^n \zz_p) \to e_k M(\zz_p /p^k \zz_p) = M_k$. This map is a composition of projection onto an idempotent summand and the restriction map induced by $p^k \zz_p \to p^n \zz_p$.
Hence we have a map $M(\zz_p /p^n \zz_p) \to \prod_{k \geqslant n} M_k$.
Taking colimits over $n$ of both sides of this map gives the structure map
$$M_\infty \to  \colim_n \prod_{k \geqslant n} M_k = \tails(M).$$
\end{proof}

The above theorem also shows that for any Mackey functor $M$, its image in $\acal$ is all that is needed to reconstruct it.

\begin{corollary}
The category of rational Mackey functors for $\zz_p$ is  equivalent to the algebraic model $\acal$.
\end{corollary}

We are very close to completing half of our task, we know by theorem \ref{thm:ssapp} that
the model category of rational $\zz_p$-spectra is Quillen equivalent to differential graded rational Mackey functors and
we have just shown that the category of rational Mackey functors is equivalent to $\acal$. But we still have to show that
the model structure on $dg \acal$ that we introduced gives the correct homotopy category, this will be easier to do once we have a better method of calculation.

\section{The Adams Short Exact Sequence}\label{sec:adams}

We define a functor from the homotopy category of
rational $\zz_p$-spectra to graded objects of $\acal(\zz_p)$.
This allows us to find a short exact sequence relating the homotopy category of spectra with $\acal(\zz_p)$,
thus allowing us to understand maps in the homotopy category of $\zz_p \Sp_\qq$.

For a spectrum $X$, the graded Mackey functor $\underline{\pi}_*(X)$ is the represented functor $[-,X]^{\zz_p}_*$ from $\ocal_{\zz_p}$ to graded rational vector spaces. From this we can make an object of $g\acal$ (the category of graded objects in $\acal$), following the recipe of the previous section.

\begin{definition}\label{def:homotopyMackey}
For a spectrum $X$, define $\underline{\pi}^{\acal}_*(X) \in g \acal$
as follows, at $n$ it takes value
$e_{n} \pi_*^{p^n \zz_p} (X)$, at infinity it takes value
$\colim_k  \pi_*^{p^k \zz_p} (X)$.
There is a map
$$
\pi_*^{p^k \zz_p} (X)
\to
\prod_{m \geqslant k}
e_{m} \pi_*^{p^m \zz_p} (X)
$$
given by restriction from $p^k \zz_p \to p^m \zz_p$
and projection onto an idempotent factor.
Applying colimits gives the structure map of $\underline{\pi}^{\acal}_*(X)$.
\end{definition}
We note that $\colim_k  \pi_*^{p^k \zz_p} (X)$
is a discrete $\zz_p$-module as any given element
is in some $\zz_p/p^k$-module. The homotopy groups above are, as usual
rationalised.
Since the functor $\underline{\pi}^{\acal}_*$ from rational $\zz_p$-spectra
to graded objects in $\acal$
is defined in terms of homotopy
groups and idempotents, it follows that
it passes to a functor on the homotopy category
of rational $\zz_p$-spectra.

\begin{lemma}\label{lem:enoughinjectives}
The monic maps in $\acal$ are exactly those maps that are
termwise injections.
If $M \in \acal$ has a surjective structure map, then it is injective.
For any object of $\acal$ there is a monomorphism to
an injective object, whose cokernel is injective.
\end{lemma}
\begin{proof}
The first statement is easily checked. For the second,
the argument of \cite[Lemma 4.2]{gre98a} is easily adapted to our setting.
%
%
%
For the third statement, let $M$ be some object of $\acal$, then define $L M_\infty$ to be that object of $\acal$ concentrated at $\infty$ where it takes value $M_\infty$. Similarly $E_n M_n$ will be the object of $\acal$ which is concentrated at $n$ and takes value $M_n$. Let $I = L(M_\infty) \prod \left( \prod_n E_n M_n \right)$, this is an injective object of $\acal$ since it is a product of injectives. One can calculate that $I_n = M_n$ and $I_\infty = M_\infty \prod \tails(M)$. There is an injective map $M \to I$, at $n$ it is given by the identity map of $M_n$ and at infinity it is given by the identity map of $M_\infty$ added to the structure map of $M$. It is clear that the cokernel of this map is concentrated at infinity and hence is injective.
\end{proof}

In particular, any object of $\acal$ or $g \acal$ that is concentrated at some $n$ (or infinity) and is zero elsewhere is injective.

\begin{lemma}\label{lem:represents}
For any spectrum $X$, the object $\underline{\pi}^{\acal}_*(e_n X)$
is injective and represents that object of $g \acal$ which
is concentrated at $n$, where it takes value
$e_{n} \pi_*^{p^n \zz_p} (X)$.

For any spectrum $X$, the object $\underline{\pi}^{\acal}_*(\hocolim_m  f_m X)$
is injective and represents that object of $\acal$ which
is concentrated at infinity, where it takes value
$\colim_n \pi_*^{p^n \zz_p} (X)$.

The functor $\underline{\pi}_*^\acal$ sends the spectrum $e_n \zz_p / p^n$ to the object $E(n) \in g \acal$ and sends $\hocolim_m f_m \zz_p / p^n$ to $L(n) \in g \acal$,  where $n \geqslant 0$. Similarly  $\underline{\pi}_*^\acal$ takes $\zz_p/p^n$ to the object $A(n)$.
\end{lemma}
\begin{proof}
The first statement is a simple calculation.
Since $e_n$ is sent to zero via the inclusion
$p^k \zz_p \to p^n \zz_p$, for $k > n$, it follows
that $\underline{\pi}^{\acal}_*(e_n X)$ has no term at infinity.
For $m \neq n$, $e_m e_n =0$, hence the only term
is at $n$ where it must take the stated value.
Injectivity is then immediate.

The second statement requires some harder calculations.
It is clear that it is concentrated at infinity, as
for large enough $m$, $f_m e_n = 0$. We are left to calculate
$$\colim_n \pi_*^{p^n \zz_p}(\hocolim_m  f_m X) =
\colim_n \colim_m f_m \pi_*^{p^n \zz_p}(X)$$
swapping the colimits, we obtain
$$\colim_m \colim_n f_m \pi_*^{p^n \zz_p}(X).$$
But for $n > m$, $f_m \in A(\zz_p)$
is sent to the identity of $f_m \in A(p^n \zz_p)$
(via the inclusion $p^n \zz_p \to \zz_p$).
Hence
$$\colim_n f_m \pi_*^{p^n \zz_p}(X) =
\colim_n \pi_*^{p^n \zz_p}(X)$$
and now taking colimits over $m$ has no effect, so they can be ignored.
Injectivity is also immediate in this case.

The calculations regarding $A(n)$, $E(n)$ and $L(n)$ are straightforward.
\end{proof}

The following result is a consequence of the Yoneda lemma, we expect the result to hold
for any spectrum $hI$ such that $\underline{\pi}_*^\acal(hI)$ is injective, but it is much easier
to see if we restrict ourselves to those that are concentrated at some $n$ or infinity.

\begin{lemma}\label{lem:injectiveiso}
Take any $X$ in $\zz_p \Sp_\qq$ and let $hI$ be some spectrum such that
$\underline{\pi}_*^\acal(hI)$ is concentrated at some $n$ or infinity.
Then $\underline{\pi}_*^\acal$ induces a natural isomorphism
$$
[X,hI]_*^{\zz_p}
\longrightarrow
\hom_\acal(\underline{\pi}_*^\acal(X), I).
$$
\end{lemma}
%
%
%

\begin{theorem}\label{thm:adams}
For any $X$ and $Y$ in $\zz_p \Sp_S$, there is a short exact sequence
$$
0 \longrightarrow
\ext^\acal_* (\underline{\pi}_*^\acal(\Sigma X), \underline{\pi}_*^\acal(Y))
\longrightarrow
[X,Y]_*^{\zz_p}
\longrightarrow
\hom_\acal(\underline{\pi}_*^\acal(X), \underline{\pi}_*^\acal(Y))
\longrightarrow 0
$$
\end{theorem}
\begin{proof}
Using lemmas \ref{lem:enoughinjectives} and \ref{lem:represents},
we can create a map $Y \to hI$ in the homotopy category of spectra,
where $hI$ is some product of terms of form $e_n X$ and $\hocolim_m f_m X$
such that
$$\underline{\pi}_*^\acal(Y) \to I = \underline{\pi}_*^\acal(hI)$$
is a monomorphism into an injective object.
If we extend this to a cofibre sequence in the homotopy category of spectra,
$$Y \to hI \to hJ$$
and let $\underline{\pi}_*^\acal(hJ) =J$,
then we also have an injective resolution of $\underline{\pi}_*^\acal(Y)$ in the algebraic model $\acal$:
$$
0 \to \underline{\pi}_*^\acal(X) \to I \to J \to 0.
$$

Applying the functor $[X,-]_*^{\zz_p}$ to the cofibre sequence produces a long exact sequence
$$
\dots \to
[X,hJ]_{n+1}^{\zz_p} \to
[X,Y]_n^{\zz_p} \to
[X,hI]_n^{\zz_p} \to
[X,hJ]_n^{\zz_p} \to
[X,Y]_{n-1}^{\zz_p} \to \dots.
$$
Using lemma \ref{lem:injectiveiso}, since $I$ is a product of objects which are
each concentrated at $n$ or infinity,  we can identify two of the above terms as maps in $\acal$:
$$[X,hI]_n^{\zz_p} = \hom_\acal(\underline{\pi}_*^\acal(X), I)_n$$
$$[X,hJ]_n^{\zz_p} = \hom_\acal(\underline{\pi}_*^\acal(X), J)_n.$$

The map
$$\hom_\acal(\underline{\pi}_*^\acal(X), I)_n \to
\hom_\acal(\underline{\pi}_*^\acal(X), J)_n$$
has kernel
$\hom_\acal(\underline{\pi}_*^\acal(X), \underline{\pi}_*^\acal(Y))_n$
and cokernel
$\ext^\acal_{n} (\underline{\pi}_*^\acal(X), \underline{\pi}_*^\acal(Y))$,
so the result follows by splitting the long exact sequence into short exact sequences.
\end{proof}

Since the objects $L(n)$ and $E(n)$ of definition \ref{def:specialAobjects} generate $\acal$, it follows that
the spectra they represent generate the homotopy category of spectra.

\begin{corollary}
The collections $e_n \zz_p / p^n$ and $\hocolim_m f_m \zz_p / p^n$ for $n \geqslant 0$ together form a collection of generators for $\zz_p \Sp_\qq$.
\end{corollary}

\section{The proof}\label{sec:proof}

We are now able to complete our classification theorem and prove that $\zz_p \Sp_\qq$ is Quillen equivalent to $dg \acal$. Rather than follow on from theorem \ref{thm:ssapp}, we give an alternative proof using the pattern of \cite{barnesfinite}, which is a method learnt from Greenlees and Shipley. In this setting we do not attempt to make a statement that involves the monoidal products, as the fibrant replacement functor of $\zz_p$-spectra isn't a monoidal functor.

We prefer this method of proof, since if one were to verify that there is a good model category of equivariant EKMM spectra for profinite groups (and one would expect every object of this model structure to be fibrant), then it would be routine to show that $dg \acal$ and $\zz_p \Sp$ are monoidally Quillen equivalent.

\begin{theorem}\label{thm:main}
The category of rational $\zz_p$-spectra is Quillen equivalent to the model category of differential graded objects in the algebraic model $\acal(\zz_p)$.
\end{theorem}
\begin{proof}
Consider our chosen collection of generators, $e_n \zz_p/p^n$ and $\colim_m f_m \zz_p/p^n$, for $n \geqslant 0$, take a fibrant replacement of each, call this set $\gcal_{top}$ and call the full spectral subcategory on this object set $\ecal_{top}$. Then $\zz_p \Sp_\qq$ is Quillen equivalent to $\rightmod \ecal_{top}$ by \cite[Theorem 3.3.3]{ss03stabmodcat}. Apply the functors of \cite{shiHZ} to obtain a $dg \qq$-enriched category $\ecal_t$ and a zig-zag of Quillen equivalences between $\rightmod \ecal_{top}$ and $\rightmod \ecal_{t}$.

There is a zig-zag of functors of enriched categories
$$\h_0 \ecal_t \leftarrow C_0 \ecal_t \rightarrow \ecal_t$$
where $(C_0 \ecal_t)_n$ is zero for negative $n$, $(\ecal_t)_n$ for $n >0$ and is the zero-cycles
of $(\ecal_t)_0$ in degree zero.
By the work of Shipley, we know that $\h_* \ecal_t = \pi_* \ecal_{top}$
and, by our calculations in theorem \ref{thm:allinzero}, that these are both
concentrated in degree zero. It follows that $\h_* \ecal_t = \h_0 \ecal_t$.
Hence, by \cite[A.1.1]{ss03stabmodcat},
this zig-zag induces Quillen equivalences on the corresponding
categories of right modules,
so $\rightmod \ecal_{t}$ is Quillen equivalent to
$\rightmod \h_0 \ecal_{t}$.

Let $\ecal_a$ be the full subcategory of $dg \acal$ on the objects $E(n)$ and $L(n)$ for $n \geqslant 0$, then
$dg \acal$ is Quillen equivalent to $\rightmod \ecal_a$.
The Adams short exact sequence of theorem \ref{thm:adams} implies that
$$\underline{\pi}_*^\acal \co \h_0 \ecal_t \longrightarrow \ecal_a$$
is an isomorphism. Hence we have shown that
$\rightmod \h_0 \ecal_{t}$ is equivalent to $\rightmod \ecal_{a}$.
Since $\rightmod \ecal_{a}$ is Quillen equivalent to $dg \acal$, we have proven the main result.
\end{proof}

\section{\texorpdfstring{Relation to $\zz_p/p^n$, $p^n \zz_p$ and $O(2)$-spectra}{Relation to subgroups, quotients and O(2)}}\label{sec:relations}

The map $\zz_p \to \zz_p/p^n$ induces an algebraic inflation map $\acal(\zz_p/p^n) \to \acal(\zz_p)$. An object of $\acal(\zz_p/p^n)$ consists of a collection of $\qq[\zz_p/p^k]$-modules, for $n \geqslant k \geqslant 0$. Let $M$ be an object of this category, with $M_k$ a $\qq[\zz_p/p^k]$-module. We inflate this to an object of $\acal(\zz_p)$ by setting $(\textrm{inf}_n M)_k = M_k$ for  $n \geqslant k \geqslant 0$ and $(\textrm{inf}_n M)_m = M_k$ for all other $m$ (including $m = \infty$). The structure map is then the diagonal. Denote this algebraic inflation map by $\varepsilon^*_a$.

We also have an algebraic restriction map coming from $p^n \zz_p \to \zz_p$, this corresponds to simply forgetting about $M_0, \dots, M_{n-1}$ and considering $M_n$ as a $\qq$-module, $M_{n+1}$ as a $\qq[p^n \zz_p/p^{n+1}]$-module and so on, with $M_\infty$ being considered as a discrete $p^n \zz_p$-module. Denote this algebraic restriction map by $i^*_a$.

\begin{proposition}
There are commutative diagrams of triangulated categories
$$
\xymatrix{
\ho \zz/p^n \Sp_\qq
\ar[r]^{\varepsilon^*}
\ar[d]^{\underline{\pi}_*^\acal} &
\ho \zz_p \Sp_\qq
\ar[d]^{\underline{\pi}_*^\acal}
&&
\ho \zz_p \Sp_\qq
\ar[r]^{i^*}
\ar[d]^{\underline{\pi}_*^\acal} &
\ho p^n \zz_p \Sp_\qq
\ar[d]^{\underline{\pi}_*^\acal} \\
\ho dg \acal(\zz/p^n)
\ar[r]^{\varepsilon^*_a} &
\ho dg \acal(\zz_p)
&&
\ho dg \acal(\zz_p)
\ar[r]^{i^*_a} &
\ho dg \acal(p^n \zz_p).
}
$$
\end{proposition}

The category $\acal$ is strikingly similar to the algebraic model of \cite{barnesdihedral},
which we now describe, starting with the Burnside ring of $O(2)$.

\begin{definition}
The space $\Phi O(2)$ can be described as a subspace of $\rr$
$$
\pscr = \{1/n | n \geqslant 1 \} \cup \{ 0, -1\}
$$
the point $1/n$ corresponds to the conjugacy class of dihedral subgroups of order $2n$,
the point $0$ corresponds to $O(2)$ and $-1$ corresponds to $SO(2)$.
The Burnside ring of $O(2)$ is then the ring of continuous maps from
$\pscr$ to $\qq$ (a discrete space).
\end{definition}

Clearly, $A(O(2))$ contains the ring of eventually constant sequences, but if we compare
$\Phi O(2)$ and $\scal \zz_p$, we see the the accumulation points behave differently.
In the first, the accumulation point is the whole group, in the second it is the trivial group.
This theme continues as we further compare the case of $O(2)$ and $\zz_p$, one appears to be the
'dual' of the other.

We let $e_n$ be the idempotent which sends $1/n$ to $1$ and all other points to zero, we
usually let $c$ denote the idempotent $e_{-1}$.
The model category of
rational $O(2)$-spectra, $O(2) \Sp_\qq$, is Quillen equivalent to
$$L_{cS} \Sp_\qq \times  L_{(1-c)S} \Sp_\qq$$
by \cite{barnesdihedral}.
We call $L_{(1-c)S} \Sp_\qq$ the model category of dihedral spectra, it is Quillen equivalent to
an algebraic model $dg \acal(\dscr)$.

\begin{definition}
An object of $\acal(\dscr)$ consists of a collection of $\qq [\zz/2 ]$-modules
$M_n$ and a $\qq$-module $M_\infty$, with a map of $\qq [\zz/2 ]$-modules
from
$$M_\infty \to \colim_{n} \prod_{k \geqslant n} M_n = \tails (M).$$
\end{definition}

The reason $\zz_2$appears is that it is the Weyl group of the dihedral group of order
$2n$ inside $O(2)$. This model is also built from $O(2)$-Mackey functors, via the formula below
$$M_n = e_n M(D_{2n}), \qquad M_\infty = \colim_{e(O(2))=1} e M(O(2)/O(2)).$$

So for $\acal(\dscr)$, there is a term at infinity which has an action of the trivial group,
whereas for $\acal(\zz_p)$ term at infinity has an action of the whole group.
That is, if $N$ is a rational $\zz_p$-Mackey functor, we obtain an object of $\acal(\zz_p$
by the use of the formulas
$$N_n  = e_n N(\zz_p /p^n ), \quad
N_ \infty  = \colim_n N(\zz_p/p^n), \quad
N_\infty^{p^k \zz_p} \cong \colim_{f(\{e\})=1} f N(\zz_p/p^k).$$

The paper \cite{gre98b} examines the behaviour of rational Mackey functors for compact Lie groups;
the comparison above demonstrates that one can expect rational Mackey functors for profinite groups
to follow many of the same rules, but in a `dual' manner. As an example, compact Lie groups
often have infinite ascending chains of closed subgroups:
$D_{2n}$ tends to $O(2)$ as $n$ tends to infinity,
but cannot have infinite descending chains of subgroups. It is these ascending chains that
give structure maps: $M_\infty \to \tails (M)$.  Profinite groups have
descending chains of open subgroups that converge to closed subgroups: $p^n \zz_p$
tends to the trivial groups as $n$ tends to infinity.  Thus we expect that
it is these descending chains that give rise to structure maps: $N_\infty \to \tails (N)$.
It will be interesting to find further instances of this dual behaviour.

\addcontentsline{toc}{part}{Bibliography}
\bibliography{profinite}
\bibliographystyle{alpha}

\end{document}